\documentclass[12pt]{amsart}
\usepackage{fullpage}

\usepackage{xypic}

\usepackage{enumerate}
\usepackage{amsthm}
\usepackage{amsmath}
\usepackage{amssymb}
\usepackage{amsfonts}
\usepackage{dsfont}
\usepackage{hyperref}
\usepackage{color}

\newtheorem{theo}{Theorem}[section]
\newtheorem{lemma}[theo]{Lemma}

\newtheorem{cor}[theo]{Corollary}

\newtheorem*{theononumber}{Theorem}
\newtheorem*{conjnonumber}{Conjecture}

\theoremstyle{definition}
\newtheorem{defi}[theo]{Definition}

\newtheorem{rem}[theo]{Remark}
 
\newcommand{\f}{\phi}

\newcommand{\gal}{\operatorname{Gal}}
\newcommand{\Gl}{\operatorname{GL}}

\newcommand{\Aut}{\operatorname{Aut}}

\newcommand{\trdeg}{\operatorname{trdeg}}

\newcommand{\N}{\mathcal{N}}

\newcommand{\de}{\partial}

\newcommand{\C}{\mathbb{C}}
\newcommand{\CC}{\mathcal{C}}

\newcommand{\rank}{\operatorname{rank}}

\newcommand{\QQ}{\ensuremath{\mathbb{Q}}}

\newcommand{\FF}{\ensuremath{\mathbb{F}}}

\title{Free differential Galois groups}
\author{Annette Bachmayr, David Harbater, Julia Hartmann and Michael Wibmer}

\date{October 20, 2020}
\thanks{The first author was funded by the Deutsche Forschungsgemeinschaft (DFG) - grants MA6868/1-1, {MA6868/1-2}. The second and third authors were supported on NSF grants DMS-1463733 and DMS-1805439.  The fourth author was supported by the NSF grants DMS-1760212, DMS-1760413, DMS-1760448 and the Lise Meitner grant M 2582-N32 of the Austrian Science Fund FWF.  We also acknowledge support from NSF grant DMS-1952694}


\begin{document}

\begin{abstract}
We study the structure of the absolute differential Galois group of a rational function field over an algebraically closed field of characteristic zero.  In particular, we relate the behavior of differential embedding problems to the condition that the absolute differential Galois group is free as a proalgebraic group.  Building on this, we prove Matzat's freeness conjecture in the case that the field of constants is algebraically closed of countably infinite transcendence degree over $\QQ$. This is the first known case of the twenty year old conjecture.

\end{abstract}

\maketitle

\section{Introduction}

In this paper, we prove the first known case of a conjecture due to Matzat on the freeness of  absolute differential Galois groups of rational geometric function fields of characteristic zero:
\begin{conjnonumber}[Matzat's conjecture] 
If $k$ is an algebraically closed field of characteristic zero, the absolute differential Galois group of $k(x)$ is the free proalgebraic group on a set of cardinality $|k|$. 
\end{conjnonumber}
Here $k(x)$ is a differential field with respect to the derivation $d/dx$.

 This conjecture was stated during the 1999 MSRI program {\em Galois Groups and Fundamental Groups} and has stymied researchers since. Implicit in the conjecture was the existence of a suitable notion of freeness for proalgebraic groups.  This notion had been introduced and studied in \cite{LuMa82} and \cite{LuMa83} in the case of prounipotent groups, but for the general case the notion was established only in \cite{Wibmer:FreeProalgebraicGroups}.

In the present paper, we obtain several equivalent conditions that characterize the freeness of an absolute differential Galois group in terms of differential embedding problems; see Theorem~\ref{theo: equivalent to have free differential Galois group}. This theorem is shown over an arbitrary differential field of characteristic zero with algebraically closed field of constants. When the differential field is countable, we obtain a particularly simple criterion (see Corollary~\ref{cor: differential Iwasawa}):

\begin{theononumber}
Let $K$ be a countable differential field with algebraically closed field of constants of characteristic zero.
Then the absolute differential Galois group of $K$ is free on a countably infinite set if and only if every differential embedding problem of finite type is solvable. 
\end{theononumber}

Combining that with the main result of \cite{BachmayrHartmannHarbaterPopLarge} on solutions to differential embedding problems over rational function fields gives the following theorem  (see Theorem~\ref{theo: Matzat's conjecture}):

\begin{theononumber}
Let $k$ be an algebraically closed field of countably infinite transcendence degree over $\QQ$. Then Matzat's conjecture holds for $k(x)$.
\end{theononumber}

Matzat's conjecture strengthens the inverse differential Galois problem, by asserting that every linear algebraic group $G$ is realizable in $|k|$ different ways (see Corollary~\ref{cor: inverse}(b)).
An affirmative answer to the inverse differential Galois problem over
$k(x)$ is known, for $k$ an algebraically closed field of
characteristic zero; i.e., every linear algebraic group over $k$ is a differential Galois group over $k(x)$. The solution was given by \cite{Tretkoff} over ${\mathbb C}(x)$
(building on Plemelj's work (\cite{Plemelj}) on Hilbert's 21st problem), and for more general algebraically closed fields of characteristic zero in \cite{Hart05} (building on work of a number of authors; see especially \cite{Kovacic:TheInverseProblemInTheGaloisTheoryOfDifferentialFields} and 
\cite{mitschisinger}). 

The results in this paper (and in  \cite{Wibmer:FreeProalgebraicGroups}) are motivated by analogous statements in ordinary Galois theory. The classical inverse Galois problem over a field $K$ asks whether every finite group $G$ is the Galois group of a Galois field extension of $K$.  This holds if $K=k(x)$ with $k$ algebraically closed: the case $k=\C$ is classical; the characteristic zero case holds by \cite[Exp.~XIII, Cor.~2.12]{SGA1}; and the general case was shown in \cite[Corollary~1.5]{Ha84}.  Going beyond this problem, the geometric Shafarevich conjecture states that the absolute Galois group of $k(x)$ is a free profinite group if $k$ is algebraically closed (in analogy, for $k=\bar \FF_p$, to the original Shafarevich conjecture in number theory, which says that the absolute Galois group of $\QQ^{\mathrm{ab}} = \QQ^{\mathrm{cycl}}$ is free). The geometric Shafarevich conjecture was proven in characteristic zero in \cite{Do64}, and in the general case in \cite{Har95} and \cite{Pop95}.  Moreover it was shown there that the absolute Galois group of $k(x)$ is free of rank $|k|$.  These proofs relied on a result of Iwasawa \cite{Iwa53} in the countable case, and a result of Melnikov-Chatzidakis in the general case (see \cite[Lemma~2.1]{Ja}); those results say that a profinite group is free of the desired rank if and only if all embedding problems are ``sufficiently solvable''.
Proalgebraic analogs of these results were proven in \cite{Wibmer:FreeProalgebraicGroups}, and we rely on those to obtain our theorems above.

A. Magid has recently shown in \cite{Magid:TheDifferentialGaloisGroupOfTheMaximalProunipotentExtensionIsFree} that the maximal prounipotent quotient of the absolute differential Galois group of any differential field $K$ of characteristic zero with algebraically closed constants is a free prounipotent group.

We should note that there is another interpretation of Matzat's conjecture. For $k$ an algebraically closed field of characteristic zero, the category of finite dimensional differential modules over $K=k(x)$ is naturally a neutral tannakian category over $k$. (See \cite[Section 9]{Deligne:categoriestannakien} or \cite[Example~B.23]{SingerPut:differential} for details.) 
The corresponding fundamental group scheme is the absolute differential Galois group of $K$ and there is an equivalence of tannakian categories between the category of finite dimensional differential modules over $K$ and the category of finite dimensional representations of the absolute differential Galois group of $K$.

According to the proof of  \cite[Theorem 2.17]{Wibmer:FreeProalgebraicGroups} the free proalgebraic group on a set $X$ can be constructed as the fundamental group of the neutral tannakian category of all cofinite representations of $F_X$. Here $F_X$ is the (abstract) free group on the set $X$, and a finite dimensional $k$-linear representation of $F_X$ is called cofinite if all but finitely many elements of $X$ act trivially. One thus obtains a tannakian reformulation of Matzat's conjecture:

\begin{conjnonumber}[Tannakian formulation of Matzat's conjecture] For $k$ an algebraically closed field of characteristic zero, the tannakian category of finite dimensional differential modules over $K=k(x)$ is equivalent to the tannakian category of all cofinite representations of $F_X$, where $X$ is a set of cardinality $|k|$.
\end{conjnonumber}

\section{Preliminaries}
\subsection{Proalgebraic groups}
In this subsection, we recall some definitions and properties of proalgebraic groups from \cite{Wibmer:FreeProalgebraicGroups}. We present the results in a form suitable for application in this paper, rather than presenting them in the greatest generality. In particular, we do not consider pro-$\mathcal C$-groups for arbitrary classes $\mathcal C$ (i.e., projective limits of groups contained in $\mathcal C$, for example pro-unipotent-groups), but only use the results in \cite{Wibmer:FreeProalgebraicGroups} for the class $\mathcal C$ of all affine group schemes of finite type over $k$.

Let $k$ be a field with algebraic closure $\bar k$. A \emph{proalgebraic group over $k$} (or more accurately a \emph{pro-affine algebraic group}) is a projective limit of affine group schemes of finite type over $k$. Projective limits exist in the category of proalgebraic groups and they can be taken pointwise, i.e., $(\varprojlim G_i)(R)=\varprojlim G_i(R)$ for any $k$-algebra $R$. The coordinate ring $k[\varprojlim G_i]$ is the direct limit of the coordinate rings $k[G_i]$. It is well-known that the concepts of ``proalgebraic groups'' and ``affine group schemes'' are equivalent. We say that a proalgebraic group is \textit{algebraic}, if it is an affine group scheme of finite type. 

Let $\phi:G \to H$ be a morphism of proalgebraic (e.g., algebraic) groups, and write $\phi(G)$ for its scheme-theoretic image (i.e., the smallest closed subscheme of $H$ through which $\phi$ factors).  We say that $\phi$ is an {\em epimorphism} if $\phi(G)=H$.  This condition holds if and only if the dual map $\phi^*:k[H] \to k[G]$ on coordinate rings is injective (or equivalently, faithfully flat); see \cite[Section~2.1]{Wibmer:FreeProalgebraicGroups}.

For a family $\phi_i \colon G_i\twoheadrightarrow H$, $i\in I$, of epimorphisms of algebraic groups, the fiber product of the groups $G_i$ over $H$ is a proalgebraic group with coordinate ring the direct limit of the rings $k[G_{i_1}]\otimes_{k[H]}\cdots \otimes_{k[H]}k[G_{i_n}]$, over all finite subsets $\{i_1,\dots, i_n\}$ of $I$ ordered by inclusion.

A closed normal subgroup $N$ of a proalgebraic group $G$ is called \textit{coalgebraic} if $G/N$ is algebraic. A set $\N$ of coalgebraic subgroups of $G$ is called a \textit{neighborhood basis at 1} for $G$ if for every coalgebraic subgroup $N$ of $G$ there exists an $N'\in \N$ with $N'\subseteq N$.

The \textit{rank} of a non-trivial proalgebraic group $G$ is defined as the smallest cardinal $\kappa$ such that there exists a neighborhood basis at 1 of cardinality $\kappa$. The rank of the trivial group is defined as zero. Note that if $G\neq 1$ is algebraic, then it is of rank one.  If $G$ is not algebraic, then the rank of $G$ equals the dimension of $k[G]$ as a $k$-vector space and it also equals the smallest cardinal $\kappa$ such that $k[G]$ can be generated as a $k$-algebra by a set of cardinality $\kappa$. In particular, $G$ is of finite rank if and only if it is algebraic (and then the rank is one or zero as noted above).

We now proceed to free proalgebraic groups. Let $X$ be a set and let $G$ be a proalgebraic group over $k$. We say that a map $\phi \colon X \to G(\bar k)$ \textit{converges to 1} if for every coalgebraic subgroup $N$ of $G$ almost all elements of $X$ map into $N(\bar k)$.  Following \cite[Def.~2.18]{Wibmer:FreeProalgebraicGroups}, a
proalgebraic group $\Gamma$ together with a map $\iota\colon X\to \Gamma(\bar k)$ that $\iota$ converges to~$1$ is called a  \textit{free proalgebraic group} on $X$ if $\Gamma$ satisfies the following universal property: For every other pair $(\Gamma',\iota')$ with these properties there exists a unique morphism $\psi\colon \Gamma\to \Gamma'$ of proalgebraic groups with $\iota'=\psi_{\bar k}\circ \iota$:
$$\xymatrix{ X \ar[rr]^\iota \ar[rd]_{\iota'} & & \Gamma(\bar k) \ar[ld]^{\psi_{\bar k}}\\
	& \Gamma'(\bar k) & 	
	}$$
The existence of such a $\Gamma$ is shown in \cite[Theorem 2.17]{Wibmer:FreeProalgebraicGroups} and it is unique (up to isomorphism) by the universal mapping property.
If $k$ has characteristic zero and $|X|\geq|k|$, the rank of $\Gamma$ is~$|X|$ (see Corollary~3.12 of \cite{Wibmer:FreeProalgebraicGroups}).

	An \emph{embedding problem} for a proalgebraic group $\Gamma$ consists of epimorphisms $G\twoheadrightarrow H$ and $\Gamma\twoheadrightarrow H$ of proalgebraic groups. A \emph{(proper) solution} is an epimorphism $\Gamma\twoheadrightarrow G$ such that 
	$$\xymatrix{
		\Gamma \ar@{..>>}[d] \ar@{->>}[rd] &  \\
		G \ar@{->>}[r] & H
	}
	$$
	commutes. In \cite{Wibmer:FreeProalgebraicGroups}, such an embedding problem is called a pro-$\mathcal C$-embedding problem and it is called a $\mathcal C$-embedding problem if $G$ (and thus also $H$) are algebraic (as before, $\mathcal C$ is the class of affine group schemes of finite type over $k$).

\subsection{Differential Galois theory} \label{subsec: dgt}
 In this subsection we recall the basics of differential Galois theory. Classic references are \cite{SingerPut:differential} and \cite{Magid:LecturesOnDifferentialGaloisTheory}. In this paper, we consider infinite families of differential equations. In particular, we want to define the absolute differential Galois group of a differential field. Some of the intermediate results are also shown for general fields of constants (not necessarily algebraically closed). Differential Galois theory in this generality is treated in \cite{AmanoMasuokaTakeuchi:HopfPVtheory} (see \cite[Cor.\ 3.5] {Takeuchi:hopfalgebraicapproach} for a proof that our definition of Picard-Vessiot extensions given below is equivalent to the Hopf-algebraic definition in \cite{AmanoMasuokaTakeuchi:HopfPVtheory}).

For the remainder of the paper, $K$ denotes a differential field of characteristic zero and $k$ its field of constants; i.e., 
 the field $K$ is equipped with a derivation $\de\colon K\to K$ and  
$$k=K^\de=\{a\in K|\ \de(a)=0\}.$$

The most important example for us is the field $K=k(x)$ of rational functions over $k$ with derivation $\de=\frac{d}{dx}$.
We are interested in linear differential equations $$\de(y)=Ay,\ A\in K^{n\times n}.$$  
Differential Galois theory associates an algebraic group to such an equation. This is achieved by first constructing a so-called \emph{Picard-Vessiot extension} of $K$ for $\de(y)=Ay$. 
More generally, one associates a proalgebaic group to a (possibly infinite) family of differential equations
\begin{equation} \label{eq: family} \de(y)=A_iy,\ A_i\in K^{n_i\times n_i},\ i\in I.\end{equation}

\begin{defi} \label{defi:PVring}
	A differential field extension $L/K$ is a {\it Picard-Vessiot extension} for the family (\ref{eq: family}) if
	\renewcommand{\theenumi}{(\roman{enumi})}
\renewcommand{\labelenumi}{(\roman{enumi})}  
	\begin{enumerate}
		\item \label{Yi} for every $i\in I$ there exists $Y_i\in\Gl_{n_i}(L)$ such that $\de(Y_i)=A_iY_i$,
		\item $L$ is generated as a field extension of $K$ by the entries of all the $Y_i$,
		\item $L^\de=K^\de$.
	\end{enumerate}
	The differential subalgebra $R$ of $L$ generated over $K$ by all the entries and the inverses of the determinants of all $Y_i$ is called a {\it Picard-Vessiot ring} for (\ref{eq: family}). A matrix $Y_i$ as in \ref{Yi} is called a {\it fundamental solution matrix} for the differential equation $\de(y)=A_iy$.	
\end{defi}

If $k=K^\de$ is algebraically closed, there exists a Picard-Vessiot extension for any family of differential equations and it is unique up to an isomorphism of differential field extensions of $K$.

An extension $L/K$ of differential fields is a \emph{Picard-Vessiot extension} if it is a Picard-Vessiot extension for some family of linear differential equations; here the family is not uniquely determined by the extension.  In Definition~\ref{defi:PVring}, the Picard-Vessiot ring $R\subseteq L$ consists of all the differentially finite elements in $L$, i.e., the elements $f$ where the $K$-space spanned by $f, \de(f),\de^2(f),\dots$ is finite-dimensional. Thus $R$ does not depend on the choice of the family of differential equations, and so we may define the differential Galois group of a Picard-Vessiot extension as follows.  

In the literature the expressions  ``Picard-Vessiot extension'' and ``Picard-Vessiot ring'' usually refer to the Picard-Vessiot extension (or ring) of a single differential equation. However, in the sequel we will use this term in the general sense of Definition \ref{defi:PVring}. Thus, a Picard-Vessiot extension need not be finitely generated as a field extension of $K$.  

\begin{defi} \label{defi:GaloisGroup}
	Let $L/K$ be a Picard-Vessiot extension and let $R\subseteq L$ be its Picard-Vessiot ring.
	The {\it differential Galois group} $\gal(L/K)$ of $L/K$ is the functor from the category of $k$-algebras to the category of groups that associates to a $k$-algebra $T$ the group of differential automorphisms $\Aut_\de(R\otimes_k T/K\otimes_k T)$ of $R\otimes_k T$ over $K\otimes_k T$.  	
\end{defi}

Here $T$ is considered a differential ring with the trivial derivation. This functor is representable; i.e., the differential Galois group $G$ is a proalgebraic group over $k$. In fact, $k[G]=(R\otimes_K R)^\de$ is the ring of constants of $R\otimes_K R$ and the canonical map 
\begin{equation} \label{eqn: torsor}
R\otimes_k k[G]\overset{\simeq}{\longrightarrow} R\otimes_K R 
\end{equation}
is an isomorphism.

\begin{rem} \label{rk: PV fin equiv}
	Let $L/K$ be a Picard-Vessiot extension with Picard-Vessiot ring $R$ and differential Galois group $G$. Then the following statements are equivalent (see \cite[Cor.\ 3.15]{AmanoMasuokaTakeuchi:HopfPVtheory}):
	\begin{enumerate}[(i)]
		\item $L$ is finitely generated as a field extension of $K$.
		\item $R$ is finitely generated as a $K$-algebra.
		\item $G$ is algebraic.
		\item $L/K$ is a Picard-Vessiot extension of a single equation.
	\end{enumerate}
\end{rem}
A Picard-Vessiot extension (or ring) satisfying the above equivalent conditions is said to be \emph{of finite type}.  Another equivalent condition is that $L/K$ is a Picard-Vessiot extension of a finite family of equations.

\begin{defi} \label{defi: absolute differential Galois group}
Assume that $k=K^\de$ is algebraically closed.  A Picard-Vessiot extension $\widetilde{K}$ of $K$ for the family of all linear differential equations over $K$ is called a \emph{complete Picard-Vessiot compositum} for $K$. The differential Galois group of $\widetilde{K}/K$ is called the {\it absolute differential Galois group} of $K$. 
\end{defi}
Note that $\widetilde{K}$ is unique (up to a $K$-$\partial$-isomorphism) since $k=K^\de$ is algebraically closed. Therefore also the absolute differential Galois group of $K$ is unique (up to an isomorphism of proalgebraic groups over $k$).

The term \emph{complete Picard-Vessiot compositum} for $K$ was introduced in \cite[Def.~3.32]{Magid:LecturesOnDifferentialGaloisTheory} and we will adhere to this convention. Contrary to the situation in ordinary Galois theory, $\widetilde{K}$ can itself admit nontrivial Picard-Vessiot extensions and so is not ``differentially closed''.
We note that $\widetilde{K}$ is a \emph{universal} Picard-Vessiot field, in the sense of \cite[Chapter 10]{SingerPut:differential}, for the category of all finite dimensional differential modules over $K$. Moreover, the corresponding \emph{universal differential Galois group}, in the sense of \cite[Chapter 10]{SingerPut:differential}, agrees with the absolute differential Galois group of $K$. The absolute differential Galois group is known for the differential fields of formal Laurent series (\cite[Section 10.3]{SingerPut:differential}) and convergent Laurent series (\cite[Section 10.4]{SingerPut:differential}) over the complex numbers.

In the case of a Picard-Vessiot extension $L/K$ of finite type, over an algebraically closed field of constants $k = K^\de$ and with Picard-Vessiot ring $R$, it is traditional to identify the differential Galois group $G$ with its set of $k$-points; i.e., with the algebraic group $\Aut_\de(R/K) = \Aut_\de(L/K)$.  For $k$ algebraically closed this can still be done even if $L/K$ is not of finite type (e.g., as in Definition~\ref{defi: absolute differential Galois group}), though in this case $G$ is a proalgebraic group.

\medskip

Let $L/K$ be a Picard-Vessiot extension with Picard-Vessiot ring $R$ and differential Galois group~$G$. Let $T$ be a $k$-algebra and $g\in G(T)$. Then the automorphism $g\colon R\otimes_k T\to R\otimes_k T$ extends to an automorphism $\widetilde{g}$ of the total ring of fractions of $R\otimes_kT$. Note that the total ring of fractions of $R\otimes_k T$ contains $L$ as a subring. We say that \emph{$a\in L$ is fixed by $g$} if $\widetilde{g}(a)=a$.

For a closed subgroup $H$ of $G$ we set
$$L^H=\{a\in L|\ a \text{ is fixed by all } g\in H(T) \text{ for all $k$-algebras }  T \}.$$

	Let $L/K$ be a Picard-Vessiot extension with differential Galois group $G$.
Then there is the following \textbf{Galois correspondence}:
	\begin{enumerate}[(a)]
		\item The maps $M\mapsto \gal(L/M)$ and $H\mapsto L^H$ are inclusion reversing bijections that are inverse to each other, between the set of all intermediate differential fields $K\subseteq M\subseteq L$ and the set of all closed subgroups $H$ of $G$.
		\item An intermediate differential field $M$ is a Picard-Vessiot extension of $K$ if and only if $\gal(L/M)$ is a normal subgroup of $G$. In this case, the restriction map $\gal(L/K)\to \gal(M/K)$ is an epimorphism and thus induces an isomorphism $\gal(L/K)/\gal(L/M)\simeq \gal(M/K)$.
		\item The fixed field $L^{G^0}$ under the connected component of the identity of $G$ is the relative algebraic closure of $K$ in $L$ .
	\end{enumerate} 
	
Let $L/K$ be a Picard-Vessiot extension with differential Galois group $G$ and consider a family of closed subgroups $H_i\subseteq G$, $i \in I$. 
Then by the Galois correspondence, 
\begin{equation}\label{lemma: intersection}
L^{\bigcap_{i\in I}H_i}=\prod_{i\in I}L^{H_i},
\end{equation}
where the right hand side indicates the field compositum in $L$.

\begin{lemma}\label{lemma: isomorphic}
	Let $L_1/K$ and $L_2/K$ be Picard-Vessiot extensions that are contained in a common overfield $L$ with no new constants, i.e.,  $L^\de=K^\de$. Then $L_1$ and $L_2$ are isomorphic as differential $K$-algebras if and only if $L_1=L_2$.
\end{lemma}

\begin{proof}
	Let $\gamma \colon L_1\to L_2$ be an isomorphism of differential $K$-algebras. Fix a family of differential equations $\de(y)=A_iy$, $i\in I$, over $K$ such that $L_1$ is generated over $K$ by the entries of fundamental solution matrices $Y_i$ for $A_i$, $i\in I$. Then for every $i \in I$, $\gamma(Y_i)$ is also a fundamental solution matrix for $A_i$. Thus $Y_i^{-1}\gamma(Y_i)$ has entries in $L^\de=k$ and so $L_1$ is generated over $K$ by the entries of $\gamma(Y_i)$ for $i\in I$ and $L_1\subseteq L_2$ follows. Similarly, $L_2\subseteq L_1$.
\end{proof}

\section{Differential embedding problems and Matzat's conjecture}

In this section we give necessary and sufficient conditions for a differential field to have free absolute differential Galois group, in the case of an algebraically closed field of constants.  This is given in Theorem~\ref{theo: equivalent to have free differential Galois group}, which uses Theorems~3.24 and~3.42 in \cite{Wibmer:FreeProalgebraicGroups} to obtain statements about Picard-Vessiot extensions.  Afterwards, in Corollary~\ref{cor: differential Iwasawa} and Theorem~\ref{theo: Matzat's conjecture}, we obtain the theorems from the introduction, thereby proving
Matzat's conjecture in the case of an algebraically closed field of constants of countably infinite transcendence degree over~${\mathbb Q}$.

In order to carry this out, we first need to introduce and characterize the rank of a Picard-Vessiot extension, and study composita of Picard-Vessiot extensions.

\subsection{The rank of Picard-Vessiot extensions}

\begin{defi}
	The {\it rank} of a Picard-Vessiot extension $L/K$, denoted by $\rank(L/K)$, is the smallest cardinal number $\kappa$ such that $L/K$ is the Picard-Vessiot extension for a family of differential equations of cardinality $\kappa$.
\end{defi}

\begin{lemma} \label{lemma: rank of PVextension}
	Let $L/K$ be a Picard-Vessiot extension with Picard-Vessiot ring $R$. Assume that $L/K$ is not of finite type. Then the following cardinal numbers are equal:
\renewcommand{\theenumi}{(\roman{enumi})}
\renewcommand{\labelenumi}{(\roman{enumi})}	
\begin{enumerate}
		\item $\rank(L/K)$, i.e., the smallest cardinal $\kappa$ such that $L/K$ is a Picard-Vessiot extension for a family of differential equations of cardinality~$\kappa$.
		\item The smallest cardinal $\kappa$ such that $L$ can be generated as a field extension of $K$ by $\kappa$ many elements.
		\item The smallest cardinal $\kappa$ such that $R$ can be generated as a $K$-algebra by $\kappa$ many elements.
		\item\label{iv} The vector space dimension of $R$ over $K$.
		\item The rank of the differential Galois group of $L/K$.
	\end{enumerate} 
\end{lemma}

\begin{proof}
Let $\kappa_1,\dots,\kappa_5$ be the cardinal numbers defined in the above five items, respectively.
If $L/K$ is a Picard-Vessiot extension for an infinite family of differential equations of cardinality $\lambda$, then $L$ can be generated as a field extension of $K$ by $\lambda$ many elements. (Namely, the entries of the corresponding fundamental solution matrices.)
Conversely, if $L/K$ can be generated by $\lambda$ many elements, say $a_i,\ i\in I$ with $|I|=\lambda$, then we can choose for every $i\in I$ a differential equation $\de(y)=A_iy$ such that $a_i\in L_i\subseteq L$, where $L_i$ is a Picard-Vessiot extension for  $\de(y)=A_iy$. So $L$ is a Picard-Vessiot extension of the family $\de(y)=A_iy$, $i\in I$.
This shows that $\kappa_1=\kappa_2$. A similar argument shows that $\kappa_1=\kappa_3$. 

Clearly, $\kappa_3\leq \kappa_4$. Conversely, if $(f_i)_{i\in I}$ generate $R$ as a $K$-algebra, then the union of the sets $\{f_{i_1}^{e_1}\cdots f_{i_r}^{e_r} \mid i_1,\dots,i_r \in I\}$ over $r\in \mathbb N$ and $e_1,\dots,e_r \in \mathbb N$ generates $R$ as a $K$-vector space. This is a countable union of sets of cardinality less or equal than $|I^r|=|I|$ and thus $\kappa_4\leq |I|=\kappa_3$. 
Finally, if $G$ is the differential Galois group of $L/K$, then $R\otimes_k k[G]\simeq R\otimes_K R$ by (\ref{eqn: torsor}) of Section~\ref{subsec: dgt}. It follows that $L\otimes_k k[G]\simeq L\otimes_K R$. Therefore the $K$-dimension of $R$ agrees with the $k$-dimension of $k[G]$, i.e., the rank of $G$. 
\end{proof}

We note that by definition $\rank(K/K)=0$, and that the rank of a non-trivial Picard-Vessiot extension $L/K$ is finite (and then equal to $1$) if and only if $L/K$ is of finite type. Furthermore, $\rank(L/K)\geq \trdeg(L/K)$ for every Picard-Vessiot extension $L/K$ that is not of finite type.  Also, $\rank(L_1/K)\leq \rank(L_2/K)$ for Picard-Vessiot extensions $L_1,L_2/K$ with $L_1\subseteq L_2$, using the characterization given in Lemma~\ref{lemma: rank of PVextension} \ref{iv}.

As in \cite[Section~3.2]{Wibmer:FreeProalgebraicGroups}, the {\em dimension} of a proalgebraic group $G$ over $k$ is the transcendence degree over $k$ of the field of fractions of $k[G^0]/{\frak a}$, where $\frak a$ is the nilradical of $k[G^0]$.  This agrees with the usual notion of dimension in the case of algebraic groups.

\begin{lemma} \label{lemma: rank and dim correspond to usual}
	Let $L/K$ be a Picard-Vessiot extension with differential Galois group $G$. Then $\rank(L/K)=\rank(G)$ and $\trdeg(L/K)=\dim(G)$.
\end{lemma}

\begin{proof}
If $L/K$ is of finite type, then $G$ is algebraic by Remark~\ref{rk: PV fin equiv}, and the first equality then holds by definition.  On the other hand, if $L/K$ is not of finite type, then
$\rank(L/K)=\rank(G)$ by Lemma~\ref{lemma: rank of PVextension}.
	
 To show that $\trdeg(L/K)=\dim(G)$, let $K\subseteq R\subseteq L$ be the Picard-Vessiot ring and let $K_1$ denote the relative algebraic closure of $K$ in $L$. Then $\trdeg(L/K)=\trdeg(L/K_1)$ and $\gal(L/K_1)=G^0$ by the Galois correspondence. We can therefore assume that $G$ is connected (and hence that $K$ is relatively algebraically closed in $L$). Then $L\otimes_k k[G]=L\otimes_K R$ is an integral domain and if $A\subseteq k[G]$ is a transcendence basis for the field of fractions of $k[G]$ over $k$, then $1\otimes A$ is a transcendence basis for the field of fractions of $L\otimes_k k[G]$ over $L$. Similarly, if $B\subseteq R$ is a transcendence basis for $L$ over $K$, then $1\otimes B$ is a transcendence basis for the field of fractions of $L\otimes_K R$ over $L$. Thus the transcendence degree of the field of fractions of $k[G]$ agrees with $\trdeg(L/K)$.
\end{proof}

Clearly $\rank(L/K)\leq|K|$ for every Picard-Vessiot extension $L/K$. Therefore, if $G$ is the absolute differential Galois group of $K$, then $\rank(G)\leq |K|$. To say that the absolute differential Galois group of $K$ is the free proalgebraic group on a set of cardinality $|K|$ is thus a precise reformulation of the idea that the absolute differential Galois group of $K$ is as big as possible.

\subsection{Composita of Picard-Vessiot extensions}

\begin{lemma} \label{lemma: Picard-Vessiot rings flat}
	Let $L_1\subseteq L_2$ be an inclusion of Picard-Vessiot extensions of $K$ and let $R_1\subseteq R_2$ denote the corresponding inclusion of Picard-Vessiot rings. Then $R_2$ is a flat $R_1$-algebra.
\end{lemma}

\begin{proof}
	Let $G_1$ respectively $G_2$ denote the corresponding differential Galois groups. Every inclusion of Hopf algebras over a field is flat (in fact faithfully flat); e.g., see \cite[Theorem 14.1]{Waterhouse:IntroductiontoAffineGroupSchemes}. Thus $k[G_2]$ is a flat $k[G_1]$-algebra. It follows that $L_2\otimes_k k[G_2]$ is a flat $L_2\otimes_k k[G_1]$-algebra. 
	Since $R_1\otimes_k k[G_1]=R_1\otimes_K R_1$ and $R_2\otimes_k k[G_2]=R_2\otimes_K R_2$ we have $L_2\otimes_k k[G_1]=L_2\otimes_K R_1$ and $L_2\otimes_k k[G_2]=L_2\otimes_K R_2$. Thus $L_2\otimes_K R_2$ is a flat $L_2\otimes_K R_1$-module. Since $L_2/K$ is faithfully flat it follows that $R_2$ is a flat $R_1$-algebra.
\end{proof}

We will need the following group theoretic characterization of linear disjointness
 of Picard-Vessiot extensions.

\begin{lemma} \label{lemma: linearly disjoint PV extensions}
	Let $L_0,L_1,L_2, L$ be Picard-Vessiot extensions of $K$ subject to the inclusions depicted in the following diagram:
	$$
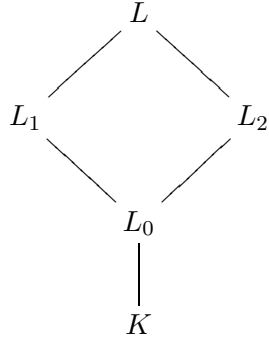
\xymatrix{
	& L \ar@{-}[ld] \ar@{-}[rd] & \\
	L_1 \ar@{-}[rd] & & L_2 \ar@{-}[ld] \\
	& L_0  \ar@{-}[d] & \\
	& K &	
	}
	$$ and let $R_i\subseteq L_i$ denote the corresponding Picard-Vessiot rings.
Then
\begin{enumerate}[(a)]
\item the field compositum $L_1L_2 \subseteq L$ is a Picard-Vessiot extension whose Picard-Vessiot ring is the ring compositum of $R_1$ and $R_2$ in $L$ (i.e., the smallest subring containing $R_1$ and $R_2$),  and
\item there is a  canonical embedding 	
\begin{equation}\label{iso}
\gal(L_1L_2/K)\to \gal(L_1/K)\times_{\gal(L_0/K)} \gal(L_2/K)
\end{equation} which is an isomorphism if and only if $L_1$ and $L_2$ are linearly disjoint over $L_0$.
\end{enumerate}
\end{lemma}

\begin{proof}
	Assertion (a) is immediate from the definitions. The embedding (\ref{iso}) is induced by the restriction homomorphisms $\gal(L_1L_2/K)\to \gal(L_i/K)$.
	
	We define $G_i=\gal(L_i/K)$ for $i=0,1,2$. Let us first show that the map $R_1\otimes_{R_0}R_2\to L_1\otimes_{L_0}L_2$ is injective. Let $S$ be the multiplicatively closed set $S=R_0\smallsetminus\{0\}$. Then $$R_1\otimes_{R_0}R_2\to S^{-1}(R_1\otimes_{R_0} R_2)= S^{-1}R_1\otimes_{L_0}{S^{-1}R_2} \hookrightarrow L_1\otimes_{L_0}L_2.$$ It thus suffices to show that every element $s\in S$ is a non-zero-divisor in $R_1\otimes_{R_0} R_2$. Multiplication with $s$ is an injective $R_0$-linear map $R_2\to R_2$. Because $R_1$ is flat over $R_0$ (Lemma \ref{lemma: Picard-Vessiot rings flat}), we see that multiplication with $s$ is also an injective map on $R_1\otimes_{R_0} R_2$.
	
	Assume that $L_1$ and $L_2$ are linearly disjoint over $L_0$. Then it follows from the above paragraph that the map $R_1\otimes_{R_0}R_2\to R_1\cdot R_2$ is an isomorphism, where $R_1\cdot R_2$ denotes the ring compositum of $R_1$ and $R_2$ in $L$. For a $k$-algebra $T$ we have differential isomorphisms $$(R_1\cdot R_2)\otimes_k T=R_1\otimes_{R_0}R_2\otimes_k T=(R_1\otimes_k T)\otimes_{(R_0\otimes_k T)}(R_2\otimes_k T),$$
	which allows us to define an inverse to $\gal(L_1L_2/K)\to \gal(L_1/K)\times_{\gal(L_0/K)} \gal(L_2/K)$. This finishes the first direction of the proof.
	
	Let us now assume that the map~(\ref{iso}) is an isomorphism. Equivalently, the comorphism 
	$$k[G_1]\otimes_{k[G_0]}k[G_2]\to k[G_{12}]$$ is an isomorphism.
	We have
	\begin{align*} (R_1\otimes_{R_0}R_2)\otimes_K (R_1\otimes_{R_0}R_2) & =(R_1\otimes_K R_1)\otimes_{(R_0\otimes_K R_0)}(R_2\otimes_K R_2) \\
	&=(R_1\otimes_k k[G_1])\otimes_{(R_0\otimes_k k[G_0])}(R_2\otimes_k k[G_2])\\
	&=(R_1\otimes_{R_0} R_2)\otimes_k (k[G_1]\otimes_{k[G_0]}k[G_2])
	\end{align*}
	and $$(R_1\cdot R_2)\otimes_K (R_1\cdot R_2)=(R_1\cdot R_2)\otimes_k k[G_{12}],$$
	where $G_{12}=\gal(L_1L_2/K)$.  Since $k[G_1]\otimes_{k[G_0]}k[G_2]\to k[G_{12}]$ is an isomorphism, the  kernel of 
	$$(R_1\otimes_{R_0} R_2)\otimes_k (k[G_1]\otimes_{k[G_0]}k[G_2])\to (R_1\cdot R_2)\otimes_k k[G_{12}]$$ is generated by its intersection with $(R_1\otimes_{R_0} R_2)\otimes _k (1\otimes_{k[G_0]} 1)$. Thus the kernel of $$f\colon (R_1\otimes_{R_0} R_2)\otimes_K (R_1\otimes_{R_0} R_2)\to (R_1\cdot R_2)\otimes_K (R_1\cdot R_2)$$ is generated by its intersection with $(R_1\otimes_{R_0} R_2)\otimes_K (1\otimes_{R_0}1)$, where we write $f$ for the composition 
	\begin{align*} (R_1\otimes_{R_0} R_2)\otimes_K (R_1\otimes_{R_0} R_2) & \to (R_1\otimes_{R_0} R_2)\otimes_k (k[G_1]\otimes_{k[G_0]}k[G_2])\\
	& \to (R_1\cdot R_2)\otimes_k k[G_{12}] \to  (R_1\cdot R_2)\otimes_K (R_1\cdot R_2).
	\end{align*}
	As $G_1$ and $G_2$ are both quotients of $G_{12}$ and the torsor isomorphisms are compatible with taking quotients (compare with \cite[Lemma 2.8]{BachmayrHarbaterHartmannWibmer:DifferentialEmbeddingProblems}), the following diagram commutes for both $i=1,2$:
	
	$$
	\xymatrix{
		R_i\otimes_K R_i \ar[d] & R_i\otimes_k k[G_i] \ar[d] \ar_-\simeq[l] \\
		(R_1\cdot R_2)\otimes_K (R_1\cdot R_2)   & (R_1\cdot R_2)\otimes_k k[G_{12}] \ar^-\simeq[l]	
	}
	$$ It follows that $f$ equals $\mu \otimes \mu$ with $\mu \colon R_1\otimes_{R_0} R_2 \to R_1\cdot R_2$ the multiplication morphism.
		
	We conclude that the kernel of $\mu \otimes \mu$ is generated by its intersection with $(R_1\otimes_{R_0} R_2)\otimes_K (1\otimes_{R_0}1)$ which is only possible if $\mu$ is injective. But then also $L_1\otimes_{L_0} L_2\to L$ is injective because $L_1\otimes_{L_0} L_2$ is a localization of $R_1\otimes_{R_0}R_2$, so $L_1$ and $L_2$ are linearly disjoint over $L_0$.
\end{proof}

\subsection{Differential embedding problems and free differential Galois groups}

We define differential embedding problems as in \cite{BachmayrHarbaterHartmannWibmer:DifferentialEmbeddingProblems}, \cite{BachmayrHarbaterHartmann:DifferentialEmbeddingProblemsOverLaurentSeriesFields} and \cite{BachmayrHartmannHarbaterPopLarge}. The only difference here is that we omit the ``of finite type assumption'' present in these articles.

\begin{defi}
	A {\it differential embedding problem} over $K$ is a pair $(\alpha\colon G\twoheadrightarrow H,\ L/K)$, where $\alpha$ is an epimorphism of proalgebraic groups and $L/K$ is a Picard-Vessiot extension with differential Galois group $H$. A {\it (proper) solution} is a Picard-Vessiot extension $M/K$ containing $L$, together with an isomorphism $G\simeq \gal(M/K)$ that identifies $\alpha\colon G\twoheadrightarrow H$ with the restriction map $\gal(M/K)\twoheadrightarrow \gal(L/K)$. 
	
	A differential embedding problem $(\alpha\colon G\twoheadrightarrow H,\ L/K)$, is {\it of finite type} if $G$ is an algebraic group. (Then necessarily $L/K$ is also of finite type.) The {\it kernel} of $(\alpha\colon G\twoheadrightarrow H,\ L/K)$ is the kernel of $\alpha$.  The embedding problem is {\it trivial} if its kernel is trivial (i.e., $\alpha$ is an isomorphism).
\end{defi}

Abusing notation, we will sometimes refer to the Picard-Vessiot extension $M/K$ as a solution to $(\alpha\colon G\twoheadrightarrow H,\ L/K)$. In particular, if $M'/K$ is a Picard-Vessiot extension containing $L$, then the \emph{compositum of all solutions of $(\alpha\colon G\twoheadrightarrow H,\ L/K)$ in $M'$} is the compositum of all Picard-Vessiot extensions $M/K$ contained in $M'$ and containing $L$ such that there exists an isomorphism $G\simeq \gal(M/K)$ that identifies $\alpha$ with the restriction map $\gal(M/K)\twoheadrightarrow \gal(L/K)$.

Now assume that the field of constants $k$ is algebraically closed. Then differential embedding problems over $K$ correspond to embedding problems for the absolute differential Galois group of $K$ as we will now explain.

Let $\widetilde{K}/K$ denote a complete Picard-Vessiot compositum for $K$ and let $\Gamma=\gal(\widetilde{K}/K)$ denote the absolute differential Galois group of $K$.

Let $H$ be a proalgebraic group. To specify an epimorphism $\beta\colon \Gamma\twoheadrightarrow H$ is equivalent to specifying a Picard-Vessiot extension $L/K$ together with an isomorphism $H\simeq \gal(L/K)$: If $\beta\colon \Gamma\twoheadrightarrow H$ is an epimorphism, then $L=\widetilde{K}^{\ker(\beta)}$ is a Picard-Vessiot extension of $K$ and the restriction map $\Gamma\twoheadrightarrow \gal(L/K)$ has kernel $\ker(\beta)$ (by the Galois correspondence). Thus $H\simeq \Gamma/\ker(\beta)\simeq \gal(L/K)$.
 Conversely, given a Picard-Vessiot extension $L/K$ together with an isomorphism $H\simeq \gal(L/K)$, there exists an embedding of $L/K$ into $\widetilde{K}/K$. The image of this embedding is unique and so there is no harm in also denoting this image by $L$. The restriction map then yields an epimorphism $\beta\colon \Gamma\twoheadrightarrow \gal(L/K)\simeq H$.
 
 Thus to specify a differential embedding problem $(\alpha\colon G \twoheadrightarrow H,\ L/K)$ over $K$ is equivalent to specifying an embedding problem $(\alpha\colon G \twoheadrightarrow H,\ \beta\colon\Gamma\twoheadrightarrow H)$ for $\Gamma$. Moreover, if $\f\colon\Gamma\twoheadrightarrow G$ is a solution of $(\alpha\colon G \twoheadrightarrow H,\ \beta\colon\Gamma\twoheadrightarrow H)$, then $M=\widetilde{K}^{\ker(\f)}$ contains $L=\widetilde{K}^{\ker{(\beta)}}$ and 
 $$
 \xymatrix{
 G \ar@{->>}^\alpha[r] \ar_-\simeq[d] & H \ar^-\simeq[d] \\
 \Gamma/\ker(\f) \ar_-\simeq[d] \ar@{->>}[r] & \Gamma/\ker(\beta) \ar^-\simeq[d] \\
 \gal(M/K) \ar@{->>}[r] & \gal(L/K)	
 }
 $$
 commutes. Conversely, if $M/K$ and $G\simeq \gal(M/K)$ constitute a solution of a differential embedding problem $(\alpha\colon G \twoheadrightarrow H,\ L/K)$, then $M/K$ embeds into $\widetilde{K}/K$ and
 $$
 \xymatrix{
 \Gamma \ar@{->>}[d] \ar@{->>}[rd] & \\
 \gal(M/K) \ar@{->>}[r] \ar_\simeq[d] & \gal(L/K) \ar^\simeq[d] \\
 G \ar@{->>}^\alpha[r] & H	
 }
 $$
commutes.

We are now prepared to provide characterizations of the freeness of the absolute differential Galois group in terms of differential embedding problems. In the following theorem, all composita are as fields, and are taken inside a fixed complete Picard-Vessiot compositum.  The seven conditions below respectively parallel the corresponding conditions in \cite[Theorem~3.24]{Wibmer:FreeProalgebraicGroups}, where the context was that of abstract proalgebraic groups (or more generally, pro-$\CC$-groups).

\begin{theo} \label{theo: equivalent to have free differential Galois group}
	Let $K$ be a differential field of cardinality $\kappa$ with an algebraically closed field of constants. Then the following statements are equivalent:
	
	\begin{enumerate}[(i)]
		\item The absolute differential Galois group of $K$ is the free proalgebraic group on a set of cardinality~$\kappa$.
		\item Every differential embedding problem $(G\twoheadrightarrow H,\ L/K)$ with $\rank(L/K)<\kappa$ and $\rank(G)\allowbreak\leq \kappa$ has a solution.
		\item Every differential embedding problem $(G\twoheadrightarrow H,\ L/K)$ with $\rank(L/K)<\kappa$ and algebraic kernel has a solution.
		\item For every differential embedding problem $(G\twoheadrightarrow H,\ L/K)$ of finite type and every Picard-Vessiot extension $M/K$ containing $L$ with $\rank(M/K)<\kappa$, there exists a solution $L'$, such that $L'$ and $M$ are linearly disjoint over $L$.
		\item For every non-trivial differential embedding problem $(\alpha\colon G\twoheadrightarrow H,\ L/K)$ of finite type and every Picard-Vessiot extension $M/K$ containing $L$ with $\rank(M/K)<\kappa$, there exists a solution $L'$, such that $L'\nsubseteq M$ and such that $\trdeg(L'M/M)>0$ if $\dim(\ker(\alpha))>0$. 
	\item For every non-trivial differential embedding problem $(\alpha\colon G\twoheadrightarrow H,\ L/K)$ of finite type, the compositum $M$ of all solutions satisfies $\rank(M/K)=\kappa$ and it satisfies $\trdeg(M/K)\allowbreak=\kappa$ if $\dim(\ker(\alpha))>0$.
	\item For every non-trivial differential embedding problem
$(G\twoheadrightarrow H,\ L/K)$ of finite type there exist $\kappa$ solutions $M_i$ such that the
$M_i$ are linearly disjoint over $L$.
	\end{enumerate}
\end{theo}

\begin{proof}
	Let $\widetilde{K}/K$ be a complete Picard-Vessiot compositum for $K$, and let $\Gamma=\gal(\widetilde{K}/K)$ denote the absolute differential Galois group of $K$.
	We will first verify that each of the conditions (i),\ldots,(vii) implies that $\rank(\Gamma)=\kappa$. Since $\rank(\Gamma)\leq|K|=\kappa$ it suffices to show that $\rank(\Gamma)\geq\kappa$.
	
Since $|k|\leq|K|$, condition (i) implies $\rank(\Gamma)=\kappa$ by Corollary 3.12 in \cite{Wibmer:FreeProalgebraicGroups}. To see that (ii) implies $\rank(\Gamma)=\kappa$, fix a proalgebraic group $G$ over $k$ with $\rank(G)=\kappa$ (such a $G$ exists by \cite[Ex.\ 3.3]{Wibmer:FreeProalgebraicGroups}).  Then (ii) implies that $G$ is a quotient of $\Gamma$ (by choosing $L=K$ and $H=1$) and thus $\rank(\Gamma)=\kappa$.

To see that (iii) implies $\rank(\Gamma)=\kappa$, fix a non-trivial algebraic group $G$ and consider the compositum $L$ of all Picard-Vessiot extensions with differential Galois group isomorphic to $G$ inside $\widetilde{K}$. If $\rank(L/K)\geq\kappa$ then also $\rank(\Gamma)\geq\kappa$. So we suppose that $\rank(L/K)<\kappa$. Then (iii) applied to the differential embedding problem $(G\times \gal(L/K)\twoheadrightarrow \gal(L/K),\ L/K)$ yields a Picard-Vessiot extension $M/K$ containing $L$. Without loss of generality we may assume that $M$ is contained in $\widetilde{K}$. The Picard-Vessiot extension $L_1/K$ that corresponds to the kernel of the projection $G\times \gal(L/K)\twoheadrightarrow G$ has differential Galois group isomorphic to $G$ but is not contained in $L$; a contradiction.

Clearly (iv)$\Rightarrow$(v), so it suffices to show that (v) implies $\rank(\Gamma)=\kappa$.

We argue in a fashion similar to what we did for (iii). Fix a non-trivial algebraic group $G$ and let $M$ be the compositum of all Picard-Vessiot extensions with differential Galois group isomorphic to $G$ inside $\widetilde{K}$. If $\rank(M/K)\geq\kappa$ we are done. If not, we can apply (v) to the differential embedding problem $(G\twoheadrightarrow 1,\ K/K)$ to find a solution $L'$ such that $L'\nsubseteq M$. This solution $L'$ is a Picard-Vessiot extension of $K$ with differential Galois group isomorphic to $G$ contained in $\widetilde{K}$ but not in $M$. This contradicts the definition of $M$.

To see that (vi) and (vii) both imply $\rank(\Gamma)=\kappa$, we can choose a non-trivial algebraic group $G$ and consider the differential embedding problem $(G\twoheadrightarrow 1,\ K/K)$. Both (vi) and (vii) imply that there exists a Picard-Vessiot extension $M$ of $K$ with $\rank(M/K)=\kappa$ (in (vii) we let $M$ be the compositum of the fields $M_i$). But then also $\rank(\Gamma)=\rank(\widetilde{K}/K)=\kappa$.

\medskip

Now that we know that $\rank(\Gamma)=\kappa\geq |k|$ in all seven cases, we can use Theorem 3.42 in \cite{Wibmer:FreeProalgebraicGroups} to see that (i) is equivalent to $\Gamma$ being saturated and that statements (ii) to (vii) are merely reformulations of the different characterizations of saturation in \cite[Theorem 3.24]{Wibmer:FreeProalgebraicGroups}: (ii) corresponds to (i) in  \cite[Theorem 3.24]{Wibmer:FreeProalgebraicGroups}, while (iii) corresponds to (ii) in Theorem \cite[Theorem 3.24]{Wibmer:FreeProalgebraicGroups}. To see that (iv) here corresponds to (iv) in \cite[Theorem 3.24]{Wibmer:FreeProalgebraicGroups}, we use Lemma~\ref{lemma: linearly disjoint PV extensions} above and Remark 3.28 in \cite{Wibmer:FreeProalgebraicGroups}. 

For the correspondence between (v) here and (v) in  \cite[Theorem 3.24]{Wibmer:FreeProalgebraicGroups}, we conform to the notation in \cite[Theorem 3.24]{Wibmer:FreeProalgebraicGroups} by setting $N=\gal(\widetilde K/M)$ and $\phi\colon \Gamma \twoheadrightarrow G$ given by $\Gamma \twoheadrightarrow \gal(L'/K)\cong G$. Using Equation~(\ref{lemma: intersection}) of Section~\ref{subsec: dgt}, we obtain
$$\phi(N)\cong N/\operatorname{ker}(\phi)\cap N=\gal(\widetilde K^{\operatorname{ker}(\phi)\cap N}/M)=\gal(\widetilde K^{\operatorname{ker}(\phi)}\widetilde K^N/M)=\gal(L'M/M).$$

To see that (vi) here corresponds to (vi) in Theorem \cite[Theorem 3.24]{Wibmer:FreeProalgebraicGroups}, we note that the intersection of kernels of solutions in (vi) in Theorem \cite[Theorem 3.24]{Wibmer:FreeProalgebraicGroups} corresponds to the compositum of all solutions in (vi) again by Equation~(\ref{lemma: intersection}).
	
Finally, to show that (vii) corresponds to (vii) in Theorem \cite[Theorem 3.24]{Wibmer:FreeProalgebraicGroups}, let $I$ be a set of cardinality $\kappa$ and consider the $I$-fold fiber product $\prod_{i\in I}(G\twoheadrightarrow H)$ of $G$ with itself over $H$. We claim that the solutions $M_i$ for $i\in I$ are linearly disjoint over $L$ if and only if the product map $\prod \phi_i \colon \Gamma\to\prod_{i\in I}(G\twoheadrightarrow H)$ is an epimorphism, where $\phi_i\colon \Gamma \twoheadrightarrow \gal(M_i/K)=G$. By \cite[Remark~3.23]{Wibmer:FreeProalgebraicGroups} that product map is an epimorphism if and only if the product maps $\prod \phi_j \colon \Gamma\to\prod_{j\in J}(G\twoheadrightarrow H)$ are epimorphisms for every finite subset $J\subseteq I$ which in turn is the case if and only if the the embedding $\gal(\prod_{j\in J}M_j/K)\to \prod_{j\in J}(\gal(M_j/K)\twoheadrightarrow H)$ given by restriction
is an epimorphism. That however is equivalent to $M_j$, $j\in J$, being linearly disjoint over $L$ by repeatedly applying Lemma \ref{lemma: linearly disjoint PV extensions} and the claim follows.
\end{proof}	

The following corollary explains how Matzat's conjecture goes beyond the solution of the inverse problem over $k(x)$. The solution of the inverse problem over $k(x)$ only tells us that every algebraic group is a differential Galois group over $k(x)$. Matzat's conjecture tells us which \emph{proalgebraic} groups are differential Galois groups over $k(x)$ and it tells us in how many different ways an algebraic group can occur as a differential Galois group.  

\begin{cor}\label{cor: inverse}
	Let $K$ be a differential field with an algebraically closed field of constants such that the absolute differential Galois group of $K$ is the free proalgebraic group on a set of cardinality $|K|$. Then:
	\begin{enumerate}[(a)]
		\item A proalgebraic group $G$ is a differential Galois group over $K$ if and only if $\rank(G)\leq |K|$. In particular, every algebraic group is a differential Galois group over $K$. 
		\item For a non-trivial algebraic group $G$, the set of isomorphism classes of Picard-Vessiot extensions with differential Galois group isomorphic to $G$ has cardinality $|K|$.
	\end{enumerate}
\end{cor}

\begin{proof}
Condition (ii) of Theorem \ref{theo: equivalent to have free differential Galois group} applied to the differential embedding problem $(G\twoheadrightarrow 1,\ K/K)$ shows that every proalgebraic group $G$ with $\rank(G)\leq |K|$ is a differential Galois group over $K$. On the other hand, $\rank(G)\leq |K|$ for every differential Galois group $G$ over $K$.  This proves the first part.

For the second part, condition (vii) of Theorem~\ref{theo: equivalent to have free differential Galois group} shows that the differential embedding problem $(G\twoheadrightarrow 1,\ K/K)$ has a set of $|K|$ solutions in $\widetilde{K}$ that are linearly disjoint, and in particular unequal.  By Lemma~\ref{lemma: isomorphic} they are non-isomorphic.  So the set of isomorphism classes of Picard-Vessiot extensions with differential Galois group isomorphic to $G$ has cardinality at least $|K|$, and hence exactly $|K|$, since the reverse inequality follows from the fact that there are only $|K|$ linear differential equations over $K$.
\end{proof}	

We note that the cardinality of the set of all isomorphism classes of Picard-Vessiot extensions with a fixed differential Galois group has been studied in detail by Kovacic in \cite{Kovacic:TheInverseProblemInTheGaloisTheoryOfDifferentialFields} for the case of solvable algebraic groups. In particular, Kovacic proved, as predicted by Matzat's conjecture, that the cardinality of the set of all isomorphism classes of Picard-Vessiot extensions of $K=k(x)$ with a fixed non-trivial connected solvable differential Galois group has cardinality $|K|$.

For a countable differential field the equivalence of (i) and (iii) in Theorem \ref{theo: equivalent to have free differential Galois group} reduces to the following corollary. Alternatively, this can be deduced from Corollary 3.43 in \cite{Wibmer:FreeProalgebraicGroups} by a translation from embedding problems to differential embedding problems. 

\begin{cor} \label{cor: differential Iwasawa}
	Let $K$ be a countable differential field with an algebraically closed field of constants. Then the absolute differential Galois group of $K$ is free on a countably infinite set if and only if every differential embedding problem of finite type over $K$ is solvable. \qed
\end{cor}

We conclude with a special case of Matzat's conjecture.

\begin{theo} \label{theo: Matzat's conjecture}
	Matzat's conjecture (as stated in the introduction) is true if the field of constants $k$ is countable and of infinite transcendence degree over $\mathbb{Q}$. In other words, Matzat's conjecture is true for the field $k=\overline{\mathbb{Q}(y_1,y_2,\ldots)}$, the algebraic closure of a field of rational functions in countably many variables over $\mathbb{Q}$.
\end{theo}

\begin{proof}
	If $k$ is countable, also $K=k(x)$ is countable. Thus, according to Corollary \ref{cor: differential Iwasawa}, it suffices to show that every differential embedding problem of finite type over $K$ is solvable. For $k$ of infinite transcendence degree this has been proved in \cite[Cor.\ 4.6]{BachmayrHartmannHarbaterPopLarge}.
\end{proof}

\medskip

\noindent Author information:

\medskip

\noindent Annette Bachmayr (n\'{e}e Maier): 
Institute of Mathematics, University of Mainz, Staudingerweg 9, 55128 Mainz, Germany.
E-mail: {\tt abachmay@uni-mainz.de}

\medskip

\noindent David Harbater: Department of Mathematics, University of Pennsylvania, Philadelphia, PA 19104-6395, USA. E-mail: {\tt harbater@math.upenn.edu}

\medskip

\noindent Julia Hartmann:  Department of Mathematics, University of Pennsylvania, Philadelphia, PA 19104-6395, USA. E-mail: {\tt hartmann@math.upenn.edu}

\medskip

\noindent Michael Wibmer: Institute of Analysis and Number Theory, Graz University of Technology, Koper\-nikusgasse 24, 8010 Graz, Austria.
E-mail: {\tt wibmer@math.tugraz.at}


\begin{thebibliography}{BHHW18}
	
	\bibitem[AMT09]{AmanoMasuokaTakeuchi:HopfPVtheory}
	Katsutoshi Amano, Akira Masuoka, and Mitsuhiro Takeuchi.
	\newblock Hopf algebraic approach to {P}icard-{V}essiot theory.
	\newblock In {\em Handbook of algebra}, volume~6, pages 127--171.
	Elsevier/North-Holland, Amsterdam, 2009.
	
	\bibitem[BHH18]{BachmayrHarbaterHartmann:DifferentialEmbeddingProblemsOverLaurentSeriesFields}
	Annette Bachmayr, David Harbater, and Julia Hartmann.
	\newblock Differential embedding problems over {L}aurent series fields.
	\newblock {\em J. Algebra}, 513:99--112, 2018.
		
	\bibitem[BHHP20]{BachmayrHartmannHarbaterPopLarge}
	Annette Bachmayr, David Harbater, Julia Hartmann, and Florian Pop.
	\newblock Large fields in differential Galois theory.
    \newblock {\em J.\ Inst.\ Math.\ Jussieu}. Published online January 27, 2020.
    \url{https://doi.org/10.1017/S1474748020000018}
    
	\bibitem[BHHW18]{BachmayrHarbaterHartmannWibmer:DifferentialEmbeddingProblems}
	Annette Bachmayr, David Harbater, Julia Hartmann, and Michael Wibmer.
	\newblock Differential embedding problems over complex function fields.
	\newblock {\em Doc. Math.}, 23:241--291, 2018.
	
	\bibitem[Del90]{Deligne:categoriestannakien}
	Pierre Deligne.
	\newblock Cat\'egories tannakiennes.
	\newblock In {\em The {G}rothendieck {F}estschrift, {V}ol.\ {II}}, volume~87 of
	{\em Progr. Math.}, pages 111--195. Birkh\"auser Boston, Boston, MA, 1990.
	
	\bibitem[Dou64]{Do64} 
	 Adrien Douady. 
	 \newblock D\'etermination d'un groupe de Galois. 
	 \newblock {\em C.R.~Acad.\ Sci.\ Paris}, 258:5305--5308, 1964.
	
	\bibitem[Gro71]{SGA1}
	Alexander Grothendieck.
	\newblock {\em Rev\^etements \'etales et groupe fondamental.}
	\newblock S{\'e}minaire de g{\'e}om{\'e}trie alg{\'e}brique du Bois Marie
 	 1960--61 ({SGA} 1). 
	\newblock Lecture Notes in Math., vol.~224, Springer, Berlin, 1971.  

	\bibitem[Har84]{Ha84}
	 David Harbater. 
	 \newblock Mock covers and Galois extensions. 
	 \newblock {\em J. Algebra}, 91:281--293, 1984. 
	 
	 \bibitem[Har95]{Har95}
	David Harbater. 
	\newblock  Fundamental groups and embedding problems in characteristic $p$. 
	\newblock In: {\em Recent developments in the inverse Galois problem} (M.~Fried, et al., eds.), 
	AMS Contemporary Mathematics Series, vol.~186, 1995, pp.~353--369.

	\bibitem[Hrt05]{Hart05}
	Julia Hartmann.
	\newblock On the inverse problem in differential Galois theory.
	\newblock {\em J.\ reine angew.\ Math.}\ 586:21--44, 2005. 

	\bibitem[Iwa53]{Iwa53}
	Kenkichi Iwasawa.
	\newblock  On solvable extensions of algebraic number fields. 
	\newblock {\em Ann.\ of Math.} (2) 58:548--572, 1953.

	\bibitem[Jar95]{Ja} 
	Moshe Jarden. 
	\newblock On free profinite groups of uncountable rank.
	\newblock In: {\em Recent developments in the inverse Galois problem} (M.~Fried, et al., eds.), 
	AMS Contemporary Mathematics Series, vol.~186, 1995, pp.~371--383.

	\bibitem[Kov69]{Kovacic:TheInverseProblemInTheGaloisTheoryOfDifferentialFields}
	Jerald Kovacic.
	\newblock The inverse problem in the {G}alois theory of differential fields.
	\newblock {\em Ann. of Math. (2)}, 89:583--608, 1969.
	
	\bibitem[LM82]{LuMa82}
	Alexander Lubotzky and Andy R.~Magid.  
	\newblock Cohomology of unipotent and prounipotent groups.
	\newblock {\em J.~Algebra}, 74:76--95, 1982.

	\bibitem[LM83]{LuMa83}
	Alexander Lubotzky and Andy R.~Magid.  
	\newblock Free prounipotent groups. 
	\newblock {\em J.~Algebra}, 80:323--349, 1983.

	\bibitem[Mag94]{Magid:LecturesOnDifferentialGaloisTheory}
	Andy R.~Magid.
	\newblock {\em Lectures on differential {G}alois theory}, volume~7 of {\em
		University Lecture Series}.
	\newblock American Mathematical Society, Providence, RI, 1994.
	
	\bibitem[Mag20]{Magid:TheDifferentialGaloisGroupOfTheMaximalProunipotentExtensionIsFree}
	Andy R.~Magid.
	\newblock The differential {G}alois group of the maximal prounipotent extension
	is free.
	\newblock arXiv:2001.09506.
	
	\bibitem[MS96]{mitschisinger}
	Claude Mitschi and Michael Singer.
	\newblock Connected linear groups as differential Galois groups. 
	\newblock {\em J.\ Algebra}, 184:333--36, 1996. 

	\bibitem[Ple08]{Plemelj}
	Josip Plemelj. 
	\newblock Riemannsche Funktionenscharen mit gegebener Monodromiegruppe. 
	\newblock {\em Monatsh.\ Math.\ Phys.}, 19:211--245, 1908.  

	\bibitem[Pop95]{Pop95}
	Florian Pop. 
	\newblock  \'Etale Galois covers of affine smooth curves. The geometric case of a conjecture of 		Shafarevich. On Abhyankar's conjecture. 
	\newblock  {\em Invent.\ Math.}, 120:555--578, 1995. 

	\bibitem[Tak89]{Takeuchi:hopfalgebraicapproach}
	Mitsuhiro Takeuchi.
	\newblock A {H}opf algebraic approach to the {P}icard-{V}essiot theory.
	\newblock {\em J. Algebra}, 122(2):481--509, 1989.
	
	\bibitem[TT79]{Tretkoff}
	Carol Tretkoff and Marvin Tretkoff.
	\newblock Solution of the inverse problem of differential Galois theory in
  	the classical case.
	\newblock {\em Amer.\ J.\ Math.}, 101:1327--1332, 1979.
	
	\bibitem[vdPS03]{SingerPut:differential}
	Marius van~der Put and Michael~F. Singer.
	\newblock {\em Galois theory of linear differential equations}, volume 328 of
	{\em Grundlehren der Mathematischen Wissenschaften [Fundamental Principles of
		Mathematical Sciences]}.
	\newblock Springer-Verlag, Berlin, 2003.
	
	\bibitem[Wat79]{Waterhouse:IntroductiontoAffineGroupSchemes}
	William~C. Waterhouse.
	\newblock {\em Introduction to affine group schemes}, volume~66 of {\em
		Graduate Texts in Mathematics}.
	\newblock Springer-Verlag, New York, 1979.
	
	\bibitem[Wib20]{Wibmer:FreeProalgebraicGroups}
	Michael Wibmer.
	\newblock Free proalgebraic groups.
	\newblock {\em \'Epijournal Geom. Alg\'ebrique} 4(1):1--36, 2020.
		
\end{thebibliography}
\end{document}